\documentclass{amsart}
\usepackage{amsmath, amsthm,fullpage, color}
\usepackage{graphicx, amsaddr}

\usepackage{float}

\newtheorem{theorem}{Theorem}
\newtheorem{corollary}[theorem]{Corollary}

\newtheorem{lemma}{Lemma}

 \title{Minimum degree condition for a graph to be knitted}

\author{Runrun Liu$^{1}$,  Martin Rolek$^{2}$,    Gexin Yu$^{1,2}$}

\address{
$^{1}$\small School of Mathematics and Statistics, Central China Normal University, Wuhan, Hubei, China.\\
$^2$\small Department of Mathematics, The College of William and Mary, Williamsburg, VA, 23185, USA.
}

\thanks{The research of the last author was supported in part by the Natural Science Foundation of China (11728102) and the NSA grant H98230-16-1-0316.}

\email{gyu@wm.edu}

 \date{}

\begin{document}
 \maketitle

\begin{abstract}
For a positive integer $k$, a graph is $k$-knitted if for each subset $S$ of $k$ vertices, and every partition of $S$ into (disjoint) parts $S_1, \ldots, S_t$ for some $t\ge 1$, one can find disjoint connected subgraphs $C_1, \ldots, C_t$ such that $C_i$ contains $S_i$ for each $i\in [t]:=\{1,2,\ldots, t\}$.   In this article, we show that if the minimum degree of an $n$-vertex graph $G$ is at least $n/2+k/2-1$ when $n\ge 2k+3$, then $G$ is $k$-knitted.  The minimum degree is sharp.  As a corollary, we obtain that $k$-contraction-critical graphs are $\left\lceil\frac{k}{8}\right\rceil$-connected.
\end{abstract}

\section{Introduction}
Linkage structure plays an important role in the study of graph minors. For an integer $k\ge 2$,  a graph $G$ is {\em $k$-linked}  if for every $2k$ vertices $u_1, \ldots, u_k, v_1, \ldots, v_k$, one can find $k$ internally disjoint paths $P_1, \ldots, P_k$ such that $P_i$ connects $u_i$ and $v_i$ for each $i\in [k]$.  Clearly,  a $k$-linked graph is $k$-connected. It has been an interesting problem to determine the function $f(k)$ such that $f(k)$-connected graphs are $k$-linked.

After a series of papers by Jung~\cite{J70}, Larman and Mani~\cite{LM74}, Mader~\cite{M67}, Robertson and Seymour~\cite{RS95}, Bollob\'as and
Thomason~\cite{BT96}, and Kawarabayashi, Kostochka, and the third author~\cite{KKY06}, it was shown by Thomas and Wollan~\cite{TW05} that $f(k)\le 10k$, which is the best current result.  Among these papers, Bollob\'as and Thomason~\cite{BT96} gave the first linear upper bound for $f$,  namely, $f(k)\leq 22k$.  In the proof of Bollob\'as and Thomason, they introduced the notion of knitted graphs.

For $1\le m\le k\le |V(G)|$, a graph $G$ is {\em $(k,m)$-knit} if for any set $S$ of $k$ vertices of $G$ and any partition $S_1, \ldots, S_t$ of $S$ into $t\ge m$ non-empty parts, $G$ contains vertex-disjoint connected subgraphs $C_1, \ldots, C_t$ such that $S_i\subseteq V(C_i)$ for $i\in [t]$.  Clearly, a $(2k, k)$-knit graph is $k$-linked.  Bollob\'as and Thomassen~\cite{BT96} proved that if a $k$-connected graph $G$ contains a minor $H$,  where $H$ is a graph with minimum degree at least $0.5(|H|+\lfloor 5k/2\rfloor-2-m)$,  then $G$ is $(k, m)$-knit.   They used this result to show that $f(k)\le 22k$.

Kawarabayashi and the third author~\cite{KY13} used the knitted property of graphs to study the connectivity of contraction-critical graphs.  A graph $G$ is {\em $k$-contraction-critical} if the chromatic number of $G$ is $k$, and any proper minor of $G$ is $(k-1)$-colorable.   The famous Hadwiger's Conjecture~\cite{H43} states that the only $k$-contraction-critical graph is $K_k$.  The connectivity of these graphs has been a crucial part in the study of Hadwiger's Conjecture and related problems.  Dirac~\cite{D60} proved that any $k$-contraction-critical graph is $5$-connected for $k \ge 5$.
Mader~\cite{M68} extended this and showed that any $6$-contraction-critical graph is $6$-connected, and any $k$-contraction-critical graph is $7$-connected for $k \ge 7$.  The first general result was found by Kawarabayashi~\cite{K07}, who proved that any $k$-contraction-critical graph is $\left\lceil\frac{2k}{27}\right\rceil$-connected.   This was improved later by Kawarabayashi and the third author~\cite{KY13}, who showed that such graphs are $\left\lceil\frac{k}{9}\right\rceil$-connected. Very recently, we~\cite{LRY18} showed that any $k$-contraction-critical graph is $8$-connected for all $k\ge 15$.

In~\cite{KY13} and~\cite{LRY18}, one needs to show a small dense graph is $k$-knitted, where a graph is {\em $k$-knitted} if it is $(k, m)$-knit for all $m \in [k]$.  The following result is used in~\cite{KY13}.

\begin{theorem}[\cite{FGKLSS03}]\label{k-ordered}
For every graph $G$ with order $n\ge 2k\ge 2$, if $d(x)+d(y)\ge n+\frac{3k-9}{2}$ for every pair of non-adjacent vertices $x$ and $y$, then $G$ is $k$-ordered, where a graph is {\em $k$-ordered} if  for every $k$ vertices of given order,  there is a cycle containing the $k$ vertices in the given order.
\end{theorem}

It is worth noting that every $k$-ordered graph is $k$-knitted.  For $n\ge 5k$,  Kostochka and the third author~\cite{KY05} showed that a graph $G$ with minimum degree at least $(n+k)/2-1$ is $k$-ordered, which is stronger than Theorem~\ref{k-ordered} for large $n$.  One can get better results than $\left\lceil\frac{k}{9}\right\rceil$-connected for $k$-contraction-critical graphs if this minimum degree condition could be applied.  However, it is not known if the minimum degree condition still holds for $n<5k$.  In this note, we show that the minimum degree condition holds for $n\ge 2k+3$ when we desire the slightly weaker conclusion that the graph is $k$-knitted instead of $k$-ordered.

\begin{theorem}\label{main-degree}
Let $G$ be a graph with $n\ge 2k+3$ vertices,  where $k\ge 5$.  If $\delta(G)\ge (n+k)/2-1$, then $G$ is $k$-knitted. Moreover, the minimum degree condition is sharp.
\end{theorem}

For the sharpness, we consider the graph $G$ whose vertex set consists of three disjoint sets $A, B, C$ with $|A|=\left\lfloor\frac{1}{2}(n-(k-2))\right\rfloor$,  $|B|=k-2$,  and $|C|=\left\lceil\frac{1}{2}(n-(k-2))\right\rceil$ so that $A\cup B$ and $C\cup B$ are cliques.  Clearly, the minimum degree of $G$ is $\left\lfloor\frac{n+k}{2}\right\rfloor-2$. If we take $S$ to be $B\cup \{x,y\}$ with $x\in A, y\in C$ and let $\{x,y\}$ be one part of the partition of $S$, then one cannot find a connected subgraph containing $x,y$ that is disjoint from the subgraphs containing other parts of the partition.

By using Theorem~\ref{main-degree} instead of Theorem~\ref{k-ordered}, one can obtain the following result.

\begin{corollary}\label{cor1}
For positive integer $k$, each $k$-contraction-critical graph is $\left\lceil\frac{k}{8}\right\rceil$-connected.
\end{corollary}

The proof of Corollary~\ref{cor1} is almost the line-by-line copy of the proof of Theorem 5 in~\cite{KY13}.   With a little more effort, one might be able to get $\left\lceil\frac{k}{6}\right\rceil$, which is claimed to be proved by Chen, Hu and Song~\cite{CHS18} by using a different degree condition.  Very recently,   Chen~\cite{C19} told us that  they also obtain the optimal minimum degree condition for all $n$ (including $n<2k+3$) and their proof is quite complicated (more than 30 pages); by using that, they claim to have a proof to show that each $k$-contraction-critical graph is $\left\lceil\frac{k}{5}\right\rceil$-connected.

\section{Proof of Theorem~\ref{main-degree}}
Let $G$ be a graph with $n\ge 2k+3$ vertices and minimum degree $\delta(G)\ge n/2+k/2-1$, and suppose $G$ is not $k$-knitted for some integer $k\ge 5$.  For a subgraph $T$ and vertex $u$ of $G$, let $d(u, T)$ be the degree of $u$ in $T$.  Then by definition, there is a set $S\subseteq V(G)$ with $|S|=k$ and a partition of $S$ into $t$ nonempty parts $S_1, \ldots, S_t$ such that $G$ does not contain disjoint connected subgraphs $C_1, \ldots, C_t$ such that $S_i\subseteq C_i$ for each $i\in [t]$.  Choose a partial $(k,t)$-knit $C=\cup_{i=1}^ t C_i$ with $S_i\subseteq C_i$ such that:\begin{itemize}
\item[(1)] $|C|<n$;
\item[(2)] subject to (1), the number of components in each $C_i$ is minimized;
\item[(3)] subject to (1) and (2), the number of vertices in $C$ is minimized.
\end{itemize}

By (1), $G-C$ is nonempty.  We will need the following.

\begin{lemma}[\cite{KY13}]\label{subtree}
Let $W$ be a graph. Let $S'$ be a subset of $V(W)$ with $|S'|\ge 2$, and let $F$ be a subtree of $W$ such that $F\supseteq S'$ and all leaves of $F$ belong to $S'$. Let $u\in W-F$, and suppose that $d(u, F)\ge |S'|+2$.  Then $W[V(F)\cup \{u\}]$ contains a subtree $F_0$ with $u\in F_0$ such that $|F_0|<|F|$, $F_0\supseteq S'$ and all leaves of $F_0$ belong to $S'$.
\end{lemma}

It follows from Lemma~\ref{subtree} that for each $u\in V(G-C)$ and every component $D$ of $C_i$,
\begin{equation}\label{eq-000}
d(u, D)\le |S_i\cap D|+1, \tag{$*$}
\end{equation} and moreover, $d(u, C)\le 1.5k$, where the maximum is attained only if each component contains exactly two vertices in $S$.  We further extend this lemma.

\begin{lemma}[vertex exchange lemma]\label{exchange}
If $u\in V(G-C)$ has $d(u,C_i)\ge |S_i|+1$ for some $i\in [t]$, then there exists $v\in N(u)\cap (C_i-S_i)$, $C_i-v+u$ is connected.  Consequently, $d(u,C_i)\le |S_i|+1$  for $u\in V(G-C)$.
\end{lemma}

\begin{proof}
We first assume that $C_i$ is connected.
We use double induction, first on $|S_i|$ then on $|C_i|$.  If $S_i$ has two vertices, then $u$ has three neighbors in $C_i$ which is a path, thus the three neighbors must be consecutive on the path, then we may take the middle neighbor as $v$.  So we let $|S_i|>2$.  As $C_i$ contains a spanning tree whose leaves are in $S_i$, we may assume that $C_i$ is a tree with more than two leaves.  For each leave $x\in C_i$, let $P_x$ be the path from $x$ to the closest vertex with degree at least $3$ or in $S_i$ (not including the closest vertex).  Then $C_i-P_x$ is a tree with $|S_i|-1$ vertices in $S$.

If $u$ has at least two neighbors on some $P_x$, then we may take $v$ to be a neighbor of $u$ not closest to $x$. As $u$ has at least $|S_i|+1\ge 3$ neighbors on $C_i$, $C_i-v+u$ is connected.  So assume that $u$ has at most one neighbor on each $P_x$.  If $u$ has exactly one neighbor on some $P_x$, then $u$ has at least $|S_i|$ neighbors on $C_i-P_x$.  By induction, there exists some vertex $v\in N(u)\cap (C_i-P_x-S_i)$ such that $(C_i-P_x)-v+u$ is connected.  Clearly, $C_i-v+u$ is connected as well.   So for each leave $x$ in $C_i$, $u$ has no neighbors in $P_x$.  Let $x'$ be the neighbor of $P_x$ on $C_i-P_x$.  Now let $S_i'=S_i-x+x'$.  Note that $d(u, C_i-P_x)\ge |S_i|+1\ge |S_i'|+1$ and $C_i-P_x$ has fewer vertices than $C_i$. By induction, there is some vertex $v\in C_i-P_x-S_i'$ such that $C_i-P_x-v+u$ is connected.  Now $C_i-v+u$ is connected since $v\not=x'$.

Now we assume that $C_i$ has more than one components.   By \eqref{eq-000}, $d(u,D)\le |D\cap S_i|+1$ for each component $D$ of $C_i$. It follows that $d(u,D_1)=|D_1\cap S_i|+1$ and $d(u,D_2)\ge |D_2\cap S_i|$ for some components $D_1$ and $D_2$ of $C_i$.  By the previous paragraph, there is a vertex $v_1\in N(u)\cap (D_1-S_i)$ such that $D_1-v_1+u$ is connected.  Replace $D_1$ with $D_1-v_1+u$, we get a new subgraph $C_i'$ containing $S_i$, which contains fewer components than $C_i$, since $D_1-v_1+u$ and $D_2$ are one component now.
\end{proof}

We may assume that $x,y\in S_1$ are not in the same component. Let $A=N(x)-C$ and $B=N(y)-C$. And $A_1=N(A)-(A\cup C)$ and $B_1=N(B)-(B\cup C)$.  We further require that
\begin{itemize}
\item[(4)] subject to (1)-(3), $A$ is nonempty, if possible;
\item[(5)] subject to (1)-(4), $B$ is nonempty, if possible; and
\item[(6)] subject to (1)-(5), $A\cup A_1\cup B\cup B_1$ is maximized.
\end{itemize}

By Lemma~\ref{exchange}, $u \in V(G) - C$ satisfies $d(u, C) \le |S| + t$, and when $d(u, C)= |S| + t$,  it must be the case that $d(u, C_1) = |S_1| + 1$, and there exists some $v \in N(u) \cap (C_1 - S_1)$ such that $C_1 - v + u$ is connected.
Making this exchange gives a knit $C'$ with fewer components than $C$ while $|C'| = |C|$, contradicting $(2)$.
Hence $d(u, C) \le |S| + t - 1 \le 1.5k - 1$ for all $u \in V(G) - C$.

\medskip

\begin{lemma}
The sets $A$ and $B$ are nonempty.
\end{lemma}

\begin{proof}
Suppose that $A$ is empty.  Let $u\in V(G)-C$.
We claim that if $d(u,C_i) = |S_i|+1$ for each $i\in[t]$, then $x$ has at least one non-neighbor in $C_i$.  As $u$ has $|S_i|+1$ neighbors in $C_i$,  by Lemma~\ref{exchange},  for some $v\in N(u)\cap (C_i-S_i)$, $C_i-v+u$ is connected.   Now $xv\not\in E(G)$, otherwise, we can make $A$ to be non-empty.

Since $d(u,C_1)\le |S_1|$ and $d(u, C_i)\le |S_i|+1$ for $i\ge 2$, there are at least $d(u,C)-|S|$ components $C_i$ with $d(u, C_i)=|S_i|+1$. Thus $x$ has at least $d(u, C) - |S| + 2$ non-neighbors in $C$ (including $x$ and $y$).
So $|G - C| \ge d(u) + 1 - d(u, C)$, and $|C| \ge d(x) + d(u, C) - |S| + 2$.
It follows that $n = |G| \ge d(u) + d(x) + 3 - k \ge 2 \delta(G) + 3 - k$, that is, $\delta(G) \le \frac{1}{2}(n + k - 3)$, a contradiction.

Now suppose that $A$ is nonempty and $B$ is empty.  If there exists $u\in V(G)-C-A$, then the above argument applies. So $V(G)-C=A$.  Similarly, if $A$ contains more than one vertex, then we may take one of them to be $u$ and apply the above argument to make $B$ nonempty, so $|A|=1$. Let $u\in A$. Then $\delta(G)\le d(u)\le d(u,C)\le 1.5k - 1$, thus $n/2+k/2-1\le 1.5k - 1$, contradicting that $n \ge 2k + 3$. 
\end{proof}

Note that we have not shown that $A$ and $B$ are disjoint.

\begin{lemma}\label{size-of-C}
$n-|C|\ge 6$.
\end{lemma}

\begin{proof}
For otherwise, let $n-|C|\le 5$.

We consider the degree of $u\in A\cup B$.  Note that $\delta(G)\le d(u)=d(u,C)+d(u,V(G)-C)\le 1.5k - 1 +d(u,V(G)-C)$ and $\delta(G)\ge (2k+3)/2+k/2-1=1.5k+0.5$.
So we have $d(u,V(G)-C)\ge 2$.  It follows that $V(G)-C$ contains at least three vertices, so $3\le n-|C|\le 5$. Take $u\in A$ and $v\in B$.  Then each of $u$ and $v$ has at least two neighbors in $G-C$.
So if $n-|C|=3$, then $u=v$ or $uv\in E(G)$, and we can add a path between  $x$ and $y$ by using at most two vertices in $G-C$, decreasing the number of components in $C$ while maintaining $|C|<n$. If $4\le n-|C|\le 5$, then $u=v$ or $uv\in E(G)$ or $u,v$ have a common neighbor in $G-C$, and we again can add a path from $x$ to $y$ by using at most three vertices in $G-C$, which decreases the number of components in $C$ while maintaining $|C|<n$, a contradiction.
\end{proof}

\begin{lemma}\label{lem4}
There is no path from $x$ to $y$ through $G-C$ of length at most $6$, and furthermore, $V(G)=C\cup A\cup B\cup A_1\cup B_1$, and the two components of $G-C$ are $A\cup A_1$ and $B\cup B_1$.
\end{lemma}

\begin{proof}
If there is a path from $x$ to $y$ through $G-C$ of length at most $6$, then we add this path to $C$. Now we add at most $5$ vertices to $C$, whose order is still less than $n$ by Lemma~\ref{size-of-C}, but decrease the number of components in $C$, a contradiction.

For the furthermore part, without loss of generality we may assume that $w\in V(G)-(A\cup B\cup A_1\cup B_1\cup C)$. Note that $w$ has no neighbors in one of $A_1$ or $B_1$, for otherwise, we may add a path from $x$ to $y$ of length $6$ to $C$, a contradiction.   We may assume that $w$ has no neighbors in $B_1$. Take $u\in B$. Then $u$ has no neighbors in $A\cup A_1$.  So $n\ge |C|+|A|+d(w,G-C)+d(u,G-C)+|\{w,u\}|\ge |C|+|A|+2\delta(G)-(d(w,C)+d(u,C))+2$.  It follows that $$d(w,C)+d(u,C)\ge 2\delta(G)+2-n+|C|+|A|\ge k+|A|+|C|.$$
So $|N(w)\cap N(u)\cap C|\ge k+|A|=|S|+|A|\ge 1 + \sum_{i=1}^t |S_i|$. By Pigeonhole Principle, for some $i\in [t]$, $|N(w)\cap N(u)\cap C_i|\ge |S_i|+1$.    It follows that from Lemma~\ref{exchange}, we can find $w'\in N(w)\cap N(u)\cap (C_i-S_i)$ such that $C_i-w'+w$ is connected.  But the replacement of $w'$ with $w$ in $C_i$ makes $A \cup A_1 \cup B\cup B_1$ larger.
\end{proof}

Take $u\in A_1$ if $A_1\not=\emptyset$, otherwise take $u\in A$. Likewise, take $v\in B_1$ if $B_1\not=\emptyset$ and take $v\in B$ if $B_1=\emptyset$.   By Lemma~\ref{lem4}, $u$ has no neighbors in $B\cup B_1$ and $v$ has no neighbors in $A\cup A_1$.  So $d(u)\le d(u,C)+|A\cup A_1|-1$ and $d(v)\le d(v,C)+|B\cup B_1|-1$. It follows that
$$2\delta(G)\le d(u)+d(v)\le d(u,C)+d(v,C)+(|A\cup A_1|+|B\cup B_1|)-2=d(u,C)+d(v,C)+(n-|C|)-2.$$
So we have $$d(u,C)+d(v,C)-|C|\ge 2 + 2\delta(G)-n\ge k=|S|.$$  It follows that $u$ and $v$ have at least $k=|S|$ common neighbors in $C$.  As $u,v$ have no common neighbors in $C_1$, by Pigeonhole Principle, for some $i\in [t]$, $|N(u)\cap N(v)\cap C_i|\ge |S_i|+1.$  It follows from Lemma~\ref{exchange} that we can find a vertex $w\in N(u)\cap N(v)\cap (C_i-S_i)$ for some $i$ such that $C_i-w+u$ and $C_i-w+v$ are connected.

Now if $w$ has a neighbor in $A\cup A_1\cup B\cup B_1-\{u,v\}$, then we can find a path from $x$ to $y$ via $w$ using at most three vertices in $A\cup A_1\cup B\cup B_1-\{u,v\}$.  Now, we have a knit $C'$ with fewer components, while maintaining $|C'|\le |C|+5<n$, a contradiction.
So all neighbors of $w$ are in $C\cup \{u,v\}$.  Replacing $w$ with $u$ in $C_i$, we get a new knit $C'$ so that $w\not\in C'$.  Clearly, the new knit $C'$ satisfies (1)-(3).  However, $$d(w,C')\ge \delta(G)-1\ge \frac{n+k}{2}-1-1\ge \frac{2k+3+k}{2}-2 = 1.5k - 0.5>\sum_{i=1}^t (|S_i|+1) - 1.$$
Therefore, for some $i\in [t]$, $d(w,C_i)>|S_i|+1$, a contradiction to Lemma~\ref{exchange}.
\medskip

{\bf Acknowledgement:}  The authors would like to thank the referees and Zi-Xia Song and Donglei Yang for their valuable comments.


\begin{thebibliography}{100}

\bibitem{BT96}
B. Bollob\'{a}s and A. Thomason, { Highly linked graphs},
{\em Combinatorica}, {\bf 16} (1996), 313--320.

\bibitem{C19}
Y. Cao, G. Chen, S. He, Z. Hu,  Knits of division I: degree condition, manuscript.



\bibitem{CHS18}
G. Chen, Z. Hu, F. Song, A new connectivity bound for knitted graphs and its application to Hadwiger’s conjecture, manuscript.


\bibitem{D52} G. Dirac,  Some theorems on abstract graphs,
\textit{Proc. London Math.
Soc.}, {\bf 2} (1952), 69--81.

\bibitem{D60}
G. A. Dirac.
Trennende Knotenpunktmengen und Reduzibilit\"at abstrakter Graphen mit Anwendung auf das Vierfarbenproblem.
\emph{J. Reine Angew. Math.}
204: 116-131, 1960.



\bibitem{H43}
H.~Hadwiger.
\newblock \"Uber eine Klassifikation der Streckencomplexe.
\newblock {\em Vierteljschr. Naturforsch. Ges. Z\"urich}, 88:133--142, 1943.

\bibitem{FGKLSS03}  R. Faudree, R. J. Gould, A. V. Kostochka, L. Lesniak,
I. Schiermeyer, and A. Saito,  Degree conditions for $k$-ordered hamiltonian graphs, \textit{J. Graph Theory},  {\bf 42} (2003), 199--210.



\bibitem{J70} H. A. Jung, Verallgemeinerung des $n$-fachen
Zusammenhangs fuer Graphen,  {\em Math. Annalen},  {\bf 187} (1970),
95--103.


\bibitem{K07}
K.~Kawarabayashi.
\newblock On the connectivity of minimum and minimal counterexamples to
  {H}adwiger's conjecture. \newblock {\em J. Combin. Theory Ser. B}, 97(1):144--150, 2007.


\bibitem{KKY06}
K. Kawarabayashi, A. Kostochka, and G. Yu,  On Degree Conditions for a Graph to be $k$-linked, {\em Combinatorics, Probability and Computing} 15 (2006), 685--694.

\bibitem{KY13}
K.~Kawarabayashi and G.~Yu.
\newblock Connectivities for $k$-knitted graphs and for minimal counterexamples to Hadwiger's Conjecture.
\newblock {\em J. Combin. Theory, Ser. B}, 103:320--326, 2013.

\bibitem{KY05}
A.~V. Kostochka and Gexin Yu.
\newblock An extremal problem for {$H$}-linked graphs.
\newblock {\em J. Graph Theory}, 50(4):321--339, 2005.


\bibitem{LM74} D. G. Larman and P. Mani, On the existence of certain
configurations within graphs and the 1-skeletons of polytopes,
{\em Proc. London Math. Soc.},  {\bf 20} (1974), 144--160.

\bibitem{LRY18}
R. Liu, M. Rolek, and G. Yu, $15$-contraction-critical graphs are $8$-connected, submitted.

\bibitem{M67}
W. Mader, Homomorphieeigenschaften und mittlere
Kantendichte von Graphen, {\em Math. Annalen},  {\bf 174} (1967),
265--268.

\bibitem{M68}
W.~Mader.
\newblock \"{U}ber trennende {E}ckenmengen in homomorphiekritischen {G}raphen.
\newblock {\em Math. Ann.}, 175:243--252, 1968.


\bibitem{RS95} N. Robertson and P. D. Seumour, Graph minors XIII.
The disjoint path problem, {\em J. Comb. Theory Ser. B},
{\bf 63} (1995), 65--110.

\bibitem{TW05}
R.~Thomas and P.~Wollan.
\newblock An improved linear edge bound for graph linkages.
\newblock {\em European J. Combin.}, 26(3-4):309--324, 2005.
\end{thebibliography}
 \end{document}